\documentclass[graybox]{svmult}

\usepackage{mathptmx}       
\usepackage{helvet}         
\usepackage{courier}        
\usepackage{type1cm}        
\usepackage{makeidx}         
\usepackage{graphicx}        
\usepackage{multicol}        
\usepackage[bottom]{footmisc}

\makeindex             

\usepackage{amsmath, amsfonts, latexsym, amssymb}
\usepackage{bm}

\newcommand{\balpha}{${\bm {\alpha}}$}
\def \Tr   {\text {\rm Tr}}
\def \diag   {\text {\rm diag}}
\def \det   {\text {\rm det}}
\def \coth   {\text {\rm coth}}

\begin{document}

\title*{The Symmetry Group of Gaussian States in $L^2 (\mathbb{R}^n)$}
\author{K. R. Parthasarathy}
\institute{K. R. Parthasarathy \at Indian Statistical Institute, 7, S. J. S. Sansanwal Marg, New Delhi - 110016, India, \email{krp@isid.ac.in}}

\maketitle

\def \balpha {${\bm {\alpha}}$}
\def \cov   {\text {\rm cov}}

\abstract{    This is a continuation of the expository article \cite{krp} with 
some new remarks. Let $S_n$ denote the set of all Gaussian states in the 
complex Hilbert space $L^2 (\mathbb{R}^n),$ $K_n$ the convex set of  all 
momentum and position covariance matrices of order $2n$ in Gaussian states 
and let $\mathcal{G}_n$ be the group of all unitary operators in  $L^2 (\mathbb{R}^n)$  conjugations by which leave $S_n$ invariant. Here we 
prove the following results. $K_n$ is a closed convex set for which a 
matrix $S$ is an extreme point if and only if $S=\frac{1}{2} L^{T} L$ 
for some $L$ in the symplectic group $Sp (2n, \mathbb{R}).$ Every element 
in $K_n$ is of the form $\frac{1}{2} (L^{T} L + M^{T} M)$ for some $L,M$ 
in $Sp (2n, \mathbb{R}).$ Every Gaussian state in $L^2 (\mathbb{R}^n)$ 
can be purified to a Gaussian state in $L^2 (\mathbb{R}^{2n}).$ Any 
element $U$  in the group $\mathcal{G}_n$ is of the form $U = \lambda 
W ({\bm {\alpha}}) \Gamma (L)$ where $\lambda$ is a complex scalar of modulus unity, ${\bm {\alpha}} \in \mathbb{C}^n,$ $L \in Sp (2n, \mathbb{R}),$  
$W({\bm {\alpha}})$ is the Weyl operator corresponding to ${\bm {\alpha}} $    and $\Gamma (L)$ is a unitary operator which implements the Bogolioubov automorphism of the Lie algebra generated by the canonical momentum and position observables induced by the symplectic linear transformation $L.$
}

\begin{acknowledgement}
The author thanks Professor R. Simon and Professor M. Krishna for several fruitful conversations on this subject during July-August 2010 and the Institute of Mathematical Sciences, Chennai for their warm hospitality. 
\end{acknowledgement}

\vskip0.1in
\noindent{\bf 2000 {\it \bf  Mathematics subject classification \quad}} 81S25; 60B15, 42A82, 81R30. 

\vskip0.1in


\section{Introduction}
\label{sec:1}
\renewcommand{\theequation}{\thesection.\arabic{equation}}

In \cite{krp} we defined a quantum Gaussian state in $L^2 (\mathbb{R}^n)$ as a state in which every real linear combination of the canonical momentum and position observables $p_1, p_2, \ldots, p_n; $   $q_1, q_2, \ldots, q_n$ has a normal distribution on the real line. Such a state is uniquely determined by the expectation values of $p_1, p_2, \ldots, p_n;$ $q_1, q_2, \ldots, q_n$ and their covariance matrix of order $2n.$ A real positive definite matrix $S$ of order $2n$ is the covariance matrix of the observables   $p_1, p_2, \ldots, p_n;$ $q_1, q_2, \ldots, q_n$ if and only if the matrix inequality
\begin{equation}
2S - i \,\,\,J \ge 0    \label{eq1.1}
\end{equation}
holds where

\begin{equation}
J = \left [ \begin{array}{cc} 0 & -I \\ I & 0 \end{array} \right ],   \label{eq1.2}
\end{equation}
the right hand side being expressed in block notation with $0$ and $I$ being of order $n \times n.$ We denote by $K_n$ the set of all possible covariance matrices of the momentum and position observables in Gaussian states so that
\begin{equation}
K_n = \left \{\left . S \right | S \,\, \mbox{ is a real symmetric matrix of order} \,\,2n \,\, {\rm and} \,\, 2 S - iJ \ge 0 \right \}.  \label{eq1.3}
\end{equation}
Clearly, $K_n$ is a closed convex set. Here we shall show that $S$ is an extreme point of $K_n$ if and only if $S = \frac{1}{2} L^T L$ for some matrix $L$ in the real symplectic matrix group
\begin{equation}
Sp (2n , \mathbb{R}) = \left \{ \left . L \right | L^T JL = J \right \} \label{eq1.4}
\end{equation}
with the superfix $T$ indicating transpose. Furthermore, it turns out that every element $S$ in $K_n$ can be expressed as 
$$S = \frac{1}{2} (L^T L + M^T M)$$
for some $L, M \in Sp (2n, \mathbb{R}).$  This, in turn implies that any Gaussian state in $L^2 (\mathbb{R}^n)$ can be purified to a pure Gaussian state in $L^2 (\mathbb{R}^{2n}).$

Let ${\bm {\alpha}} \in (\alpha_1, \alpha_2, \ldots, \alpha_n)^T \in \mathbb{C}^n,$ $L = (( \ell_{ij})) \in Sp (2n, \mathbb{R})$ and let $\alpha_j = x_j + iy_{j}$  with $x_j, y_j \in \mathbb{R}.$ Define a new set of momentum and position observables $p_1^{\prime}, p_2^{\prime}, \ldots, p_n^{\prime};$ $q_1^{\prime}, q_2^{\prime}, \ldots, q_n^{\prime}$ by
\begin{eqnarray*}
p_i^{\prime} & =& \sum_{j=1}^n \left \{\ell_{ij} (p_j - x_j) + \ell_{i n +j} (q_j - y_j)  \right \}, \\
q_i^{\prime} &=& \sum_{j=1}^n \left \{\ell_{n+i \,j} (p_j - x_j) + \ell_{n+i \,\,n+j} (q_j - y_j)  \right \},
\end{eqnarray*}
 for $1 \leq i \leq n.$ Here one takes linear combinations and their respective closures to obtain $p_i^{\prime}, q_i^{\prime}$ as selfadjoint operator observables. Then $p_i^{\prime}, p_2^{\prime}, \ldots, p_n^{\prime};$ $q_i^{\prime}, q_2^{\prime}, \ldots, q_n^{\prime}$ obey the canonical commutation relations and thanks to the Stone-von Neumann uniqueness theorem there exists a unitary operator $\Gamma ({\bm {\alpha}}, L)$ satisfying
\begin{eqnarray*}
p_i^{\prime} &=& \Gamma ({\bm {\alpha}}, L) \,\,  p_i \,\, \Gamma  ({\bm {\alpha}}, L)^{\dagger}, \\
q_i^{\prime} &=& \Gamma ({\bm {\alpha}}, L)  \,\, q_i  \,\, \Gamma  ({\bm {\alpha}}, L)^{\dagger}
\end{eqnarray*}
for all $1 \leq i \leq n.$ Furthermore, such a $ \Gamma ({\bm {\alpha}}, L)$ is unique upto a scalar multiple of modulus unity. The correspondence $ ({\bm {\alpha}}, L) \rightarrow \Gamma ({\bm {\alpha}}, L)$ is a projective unitary and irreducible representation of the semidirect product group $\mathbb{C}^n \,\, \textcircled{S} \,\, Sp (2n, \mathbb{R}).$ Here  any element $L$ of $Sp (2n, \mathbb{R})$  acts on $\mathbb{C}^n$ real-linearly preserving the imaginary part of the scalar product. The operator $\Gamma ({\bm {\alpha}}, L)$ can be expressed as the product of $W ({\bm {\alpha}}) = \Gamma ({\bm {\alpha}}, I)$ and $\Gamma (L) = \Gamma ({\bm {0}}, L).$ Conjugations by $W ({\bm {\alpha}})$ implement translations of $p_j, q_j$ by scalars whereas conjugations by $\Gamma(L)$ implement symplectic linear transformations by elements of $Sp (2n, \mathbb{R}),$ which are the so-called Bogolioubov automorphisms of canonical commutation relations. In the last section we show that every  unitary operator $U$ in $L^2 (\mathbb{R}^n),$ with the property that $U \rho U^{\dagger}$ is a Gaussian state whenever $\rho$ is a Gaussian state, has the form $U = \lambda W ({\bm {\alpha}}) \Gamma (L)$ for some scalar $\lambda$ of modulus unity, a vector ${\bm {\alpha}}$ in $\mathbb{C}^n$ and a matrix $L$ in the group $Sp (2n, \mathbb{R}).$ 

The following two natural problems  that arise in the context of our note seem to be open. What is the most general unitary operator $U$ in $L^2 (\mathbb{R}^n)$ with the property that whenever $| \psi \rangle$ is a  pure Gaussian state so is $U | \psi \rangle?$ Secondly, what is the most general trace-preserving and completely positive linear map $\Lambda$  on the ideal of trace-class operators on $L^2 (\mathbb{R}^n)$ with the property that $\Lambda (\rho)$ is a Gaussian state whenever $\rho$ is a Gaussian state?

\section{Exponential vectors, Weyl operators, second quantization and the quantum Fourier transform}
\label{sec:2}
\setcounter{equation}{0}

For any ${\bm {z}} = (z_1, z_2, \ldots, z_n)^T$ in $\mathbb{C}^n$ define the associated {\it exponential vector} $e ({\bm {z}})$ in $L^2 (\mathbb{R}^n)$  by
\begin{equation}
e ({\bm {z}}) ({\bm {x}}) =  (2 \pi)^{-n/4} \exp \sum_{j=1}^n ({\bm {z}}_j {\bm {x}}_j - \frac{1}{2} {\bm {z}}_j^2 - \frac{1}{4}{\bm {x}}_j^2 ).   \label{eq2.1}
\end{equation}
Writing scalar products in the Dirac notation we have
\begin{eqnarray}
\langle e ({\bm {z}}) | e ({\bm {z}}^{\prime}) \rangle &=& \exp \langle {\bm {z}} | {\bm {z}}^{\prime} \rangle \nonumber \\
 &=& \exp \sum_{j=1}^n \bar{z}_j  z_j^{\prime}. \label{eq2.2}
\end{eqnarray}
The exponential vectors constitute a linearly independent and total set in the Hilbert space $L^2 (\mathbb{R}^n).$ If $U$ is a unitary matrix of order $n$ then there exists a unique unitary $\Gamma(U)$ in $L^2 (\mathbb{R}^n)$ satisfying
\begin{equation}
\Gamma (U) | e ({\bm {z}}) \rangle = | e (U {\bm {z}})\rangle  \quad \forall \,\,  {\bm {z}} \in \mathbb{C}^n.  \label{eq2.3}
\end{equation}
The operator $\Gamma(U)$ is called the {\it second quantization} of $U.$ For any two unitary matrices $U,V$ in the unitary group $\mathcal{U}(n)$ one has
$$\Gamma (U) \Gamma(V) = \Gamma(UV).$$
The correspondence $U \rightarrow \Gamma(U)$ is a strongly continuous unitary representation of the group $\mathcal{U}(n)$ of all unitary matrices of order $n.$ 

For any $ {\bm {\alpha}} \in \mathbb{C}^n$ there is a unique unitary operator $W ( {\bm {\alpha}})$ in $L^2 (\mathbb{R}^n)$ satisfying

\begin{equation}
W  ({\bm {\alpha}}) \,\,|e ( {\bm {z}}) \rangle = e^{-\frac{1}{2} \|{\bm {\alpha}}\|^{2} - \langle   {\bm {\alpha}} |  {\bm {z}} \rangle  } \,\, | e ( {\bm {z}} +  {\bm {\alpha}}  ) \rangle \quad \forall \quad   {\bm {z}} \in \mathbb{C}^n.\label{eq2.4}
\end{equation}
For any $ {\bm {\alpha}},  {\bm {\beta}}$ in $\mathbb{C}^n$ one has 
\begin{equation}
W ( {\bm {\alpha}}) \,\,W  ({\bm {\beta}}) = e^{-i \, {\rm Im} \langle {\bm {\alpha}} |  {\bm {\beta}} \rangle }  \,\,W ( {\bm {\alpha}}+{\bm {\beta}}  ). \label{eq2.5}
\end{equation}
The correspondence $ {\bm {\alpha}} \rightarrow W ( {\bm {\alpha}})$ is a projective unitary and irreducible representation of the additive group $\mathbb{C}^n.$ The operator $W( {\bm {\alpha}})$ is called the {\it Weyl operator} associated with  ${\bm {\alpha}}.$ As a consequence of  \eqref{eq2.5} it follows that the map $t \rightarrow W (t  {\bm {\alpha}}),$ $t \in \mathbb{R}$ is a strongly continuous one parameter unitary group admitting a selfadjoint Stone generator $p({\bm {\alpha}})$ such that 
\begin{equation}
W (t  {\bm {\alpha}}) = e^{-it p ( {\bm {\alpha}} )} \quad \forall \quad   {\bm {\alpha}} \in \mathbb{C}^n. \label{eq2.6}
\end{equation}
Writing $ {\bm {e}}_j = (0,0, \ldots, 0, 1, 0, \ldots, 0)^T $ with $1$ in the $j$-th position,
 \begin{eqnarray}
p_j &=& 2^{-\frac{1}{2}} \,\,p( {\bm {e}}_j), \quad q_j = -2^{-\frac{1}{2}} p (i{\bm {e}}_j) \label{eq2.7}   \\
a_j &=& \frac{q_j + i p_j}{\sqrt{2}}, \quad a_j^{\dagger} = \frac{q_j - ip_j}{\sqrt{2}}  \label{eq2.8}
\end{eqnarray}
one obtains a realization of the momentum and position observables $p_j, q_j, 1 \leq i \leq n$ obeying  the canonical commutation relations (CCR)
$$[p_i, p_j] = 0, \quad [q_i, q_j] = 0, \quad [q_r, p_s] = i \delta_{rs} $$
and the adjoint pairs $a_j,$ $a_j^{\dagger}$ of annihilation and creation operators satisfying
$$[a_i, a_j] = 0, \quad [a_i, a_j^{\dagger}] = \delta_{ij}$$
in appropriate domains. If we write
$$p_j^s = 2^{-\frac{1}{2}} p_j, \quad  q_j^s = 2^{\frac{1}{2}} q_j$$
one obtains the canonical Schr\"odinger pairs of momentum and position observables in the form
$$\left (p_j^s \psi \right ) ({\bm {x}}) = \frac{1}{i} \,\,\frac{\partial \psi}{\partial x_j} ({\bm {x}}), \left ( q_j^s \psi \right ) ({\bm {x}}) = x_j \,\,\psi ({\bm {x}})$$
in appropriate domains. We refer to \cite{krP} for more details.

We now introduce the sympletic group $Sp (2n, \mathbb{R})$ of real matrices of order $2n$ satisfying \eqref{eq1.4}. Any element of this group is called a symplectic matrix. As described in \cite{dms}, \cite{krp}, for any symplectic matrix $L$ there exists a unitary operator $\Gamma (L)$ satisfying
\begin{equation}
\Gamma(L) \,\,W( {\bm {\alpha}}) \,\,\Gamma(L)^{\dagger} = W (\tilde{L} {\bm {\alpha}} )  \quad \forall \quad  {\bm {\alpha}} \in \mathbb{C}^n \label{eq2.9}
\end{equation}
where
\begin{equation}
\left [\begin{array}{cc} {\rm Re} & \tilde{L}   {\bm {\alpha}} \\ {\rm Im} & \tilde{L}  {\bm {\alpha}} \end{array} \right ] = L \,\,\left [\begin{array}{cc}  {\rm Re} &   {\bm {\alpha}} \\ {\rm Im} &  {\bm {\alpha}} \end{array} \right ].  \label{eq2.10}
\end{equation}
Whenever the symplectic matrix $L$ is also a real orthogonal matrix then $\tilde{L}$ is a unitary matrix and $\Gamma(L)$ coincides with the second quantization $\Gamma (\tilde{L})$ of $\tilde{L}.$ Conversely, if $U$ is a unitary matrix of order $n,$ $L_U$ is the matrix satisfying
$$L_U \left [\begin{array}{c}  {\bm {x}} \\   {\bm {y}} \end{array} \right ]  =   \left [\begin{array}{cc} {\rm Re} & U ( {\bm {x}}+i  {\bm {y}} ) \\ {\rm Im} & U ( {\bm {x}} + i  {\bm {y}}    ) \end{array} \right ]$$ 
then $L_U$ is a symplectic and real orthogonal matrix of order $2n$ and $\Gamma(L_U) = \Gamma (U).$ Equations \eqref{eq2.9} and \eqref{eq2.6} imply that $\Gamma(L)$  implements the Bogolioubov automorphism determined by the symplectic matrix $L$ through conjugation.

For any state $\rho$ in $L^2 (\mathbb{R}^n)$ its {\it quantum Fourier transform} $\hat{\rho}$ is defined to be the complex-valued function on $\mathbb{C}^n$ given by
 
 \begin{equation}
\hat{\rho} ( {\bm {\alpha}}) = \Tr \,\,\rho \,W ( {\bm {\alpha}} ), \quad   {\bm {\alpha}} \in \mathbb{C}^n.\label{eq2.11}
\end{equation}
In \cite{krp} we have described a necessary and sufficient condition for a complex-valued function $f$ on $\mathbb{C}^n$ to be the quantum Fourier transform of a state in $L^2 (\mathbb{R}^n).$  Here we shall briefly describe an inversion formula for reconstructing $\rho$ from $\hat{\rho}.$ To this end we first observe that \eqref{eq2.11} is well defined whenever $\rho$ is any trace-class operator in $L^2 (\mathbb{R}).$  Denote by $\mathcal{F}_1$ and $\mathcal{F}_2$  respectively the ideals of trace-class and Hilbert-Schmidt operators in $L^2 (\mathbb{R}^n).$ Then $\mathcal{F}_1 \subset \mathcal{F}_2$ and $\mathcal{F}_2$ is a Hilbert space with the inner product $\langle A | B \rangle = \Tr \, A^{\dagger} B.$ There is a natural isomorphism between $\mathcal{F}_2$ and  $L^2 (\mathbb{R}^n)  \otimes L^2 (\mathbb{R}^n),$ which can, in turn, be identified with the Hilbert space of square integrable functions of two variables $ {\bm {x}},  {\bm {y}}$ in $\mathbb{R}^n.$ We denote this isomorphism by $\mathcal{I}$ so that $\mathcal{I} (A) ( {\bm {x}},  {\bm {y}})$ is a square integrable function of $( {\bm {x}},   {\bm {y}}  )$ for any $A \in \mathcal{F}_2$ and

\begin{equation}
\mathcal{I}   (| e ( {\bm {u}} )   \rangle \langle e  (\bar{{\bm {v}}}  )  |)   ( {\bm {x}},  {\bm {y}} ) = e ( {\bm {u}})  ( {\bm {x}}) e ( {\bm {v}}) ( {\bm {y}} )   \label{eq2.12}
\end{equation}
for all $ {\bm {u}},$ $  {\bm {v}} \in \mathbb{C}^n,$ $\bar{ {\bm {v}}}$ denoting $(\bar{v}_1, \bar{v}_2, \ldots, \bar{v}_n ).$  From \eqref{eq2.4} and \eqref{eq2.11} we have

\begin{eqnarray*}
(| e ( {\bm {u}} ) \rangle \langle e (\bar{ {\bm {v}} })|)^{\wedge} ( {\bm {\alpha}} ) &=& \langle  e (\bar{{\bm {v}}  }) \left | W ({\bm {\alpha}} ) \right | e ({\bm {u}})   \rangle \\
&=& \exp  \left \{-\frac{1}{2} \| {\bm {\alpha}} \|^2   - \langle  {\bm {\alpha}} |  {\bm {u}} \rangle + \langle  \bar{{\bm {v}}} | {\bm {\alpha}} \rangle + \langle {\bm {v}} |   {\bm {u}} \rangle         \right \}.
\end{eqnarray*}
Substituting $ {\bm {\alpha}} =  {\bm {x}} + i   {\bm {y}} $ and using \eqref{eq2.1}, the equation above, after some algebra, can be expressed as
\begin{equation}
(| e ( {\bm {u}} \rangle \langle e ( \bar{{\bm {v}}})| )^{\wedge} ( {\bm {x}} + i  {\bm {y}}) = (2 \pi)^{n/2} e ( {\bm {u}}^{\prime} )  (\sqrt{2} {\bm {x}})  e ({\bm {v}}^{\prime})  (\sqrt{2} {\bm {y}})  \label{eq2.13}
\end{equation}
where
\begin{eqnarray}
\left [\begin{array}{c}  {\bm {u}}^{\prime}  \\   {\bm {v}}^{\prime} \end{array} \right ] &=& U \left [\begin{array}{c} {\bm {u}} \\ {\bm {v}} \end{array} \right ],  \nonumber \\
U &=& 2^{-1/2} \left [ \begin{array}{cc}  -I & I \\ iI & iI \end{array}  \right ].  \label{eq2.14}
\end{eqnarray}
Let $D_{\theta}, \theta > 0$ denote the unitary dilation operator in $L^2 (\mathbb{R}^n) \otimes L^2 (\mathbb{R}^n)$ defined by

\begin{equation}
(D_{\theta} f) ( {\bm {x}},  {\bm {y}}) = \theta^{n} f (\theta {\bm {x}}, \theta {\bm {y}}).  \label{eq2.15}
\end{equation}
Then \eqref{eq2.13} can be expressed as
$$(| e ({\bm {u}}) \rangle \langle e (\bar{{\bm {v}}}|)^{\wedge} ({\bm {x}} + i {\bm {y}}) = \pi^{n/2} \left \{D_{\sqrt{2}} \Gamma (U) e ({\bm {u}} \otimes {\bm {v}})  \right \} ({\bm {x}}, {\bm {y}})  $$
where $\Gamma (U)$ is the second quantization operator in $L^2 (\mathbb{R}^{2n})$  associated with the unitary matrix $U$ in \eqref{eq2.14} of order $2n.$ Since exponential vectors are total and $D_{\sqrt{2}}$ and $\Gamma(U)$ are unitary we can express the quantum  Fourier transform $\rho \rightarrow \hat{\rho} ({\bm {x}} + i {\bm {y}})$ as
 \begin{equation}
\hat{\rho} = \pi^{n/2} D_{\sqrt{2}} \,\,\Gamma(U) \,\,\mathcal{I} (\rho).  \label{eq2.16}
\end{equation}
In particular, $\hat{\rho} ({\bm {x}} + i {\bm {y}})$ is a square integrable function of $({\bm {x}}, {\bm {y}}) \in \mathbb{R}^n \times \mathbb{R}^n$ and 
\begin{equation}
\rho = \pi^{-n/2} \,\,\mathcal{I}^{-1} \,\,\Gamma(U^{\dagger}) \,\,D_{2^{-1/2}} \,\,\hat{\rho}  \label{eq2.17}
\end{equation}
is the required inversion formula for the quantum Fourier transform. It is a curious but an elementary fact that the eigenvalues of $U$ in \eqref{eq2.14} are all 12$^{\rm th}$ roots of unity and hence the unitary operators $\Gamma(U)$ and $\Gamma(U^{\dagger})$ appearing in \eqref{eq2.16} and \eqref{eq2.17} have their 12-th powers equal to identity. This may be viewed as a quantum analogue of the classical fact that the 4-th power of the unitary Fourier transform in $L^2 (\mathbb{R}^n)$ is equal to identity.

\section{Gaussian states and their covariance matrices}
\label{sec:3}
\setcounter{equation}{0}

We begin by choosing and fixing the canonical momentum and position observables $p_1, p_2, \ldots, p_n;$ $q_1, q_2, \ldots, q_n$ as in equation \eqref{eq2.7} in terms of the Weyl operators. They obey the CCR. The closure of any real linear combination of the form $\sum\limits_{j=1}^n (x_j p_j - y_j q_j)$ is selfadjoint and we denote the resulting observable by the same symbol. As in \cite{krp}, for  ${\bm {\alpha}} = (\alpha_1, \alpha_2, \ldots, \alpha_n)^T,$ $\alpha_j = x_j + i y_j$ with $x_j,$ $y_j \in \mathbb{R},$ the Weyl operator $W ({\bm {\alpha}})$ defined in Section \ref{sec:2} can be expressed as
\begin{equation}
W ({\bm {\alpha}}) = e ^{-i \sqrt{2} \sum\limits_{j=1}^{n} (x_{j} p_{j} - y_{j} q_{j}  ) }.  \label{eq3.1}
\end{equation}
Sometimes it is useful to express $W ({\bm {\alpha}})$ in terms of the annihilation and creation operators defined by \eqref{eq2.8}: 

\begin{equation}
W ({\bm {\alpha}}) = e^{\sum\limits_{j=1}^{n} (\alpha_j a_{j}^{\dagger} - \bar{\alpha}_{j} a_{j}   )}  \label{eq3.2}
\end{equation}
where the linear combination in the exponent is the closed version. A state $\rho$ in $L^2 (\mathbb{R})$  is said to be {\it Gaussian} if every observable of the form $\sum\limits_{j=1}^n (x_j p_j - y_j q_j)$ has a normal distribution on the real line in the state $\rho$ for $x_j, y_j \in \mathbb{R}.$ From \cite{krp} we have the following theorem.

\begin{theorem} 
\label{thm3.1} 

A state $\rho$ in $L^2 (\mathbb{R}^n)$ is Gaussian if and only if its quantum Fourier transform $\hat{\rho}$ is given by

\begin{eqnarray}
\hat{\rho} ({\bm {\alpha}}) &=& \Tr \,\rho \,W({\bm {\alpha}}) \nonumber \\
&=& \exp \left \{ - i \sqrt{2} \left ({\bm {\ell}}^T {\bm {x}} - {\bm {m}}^T {\bm {y}} \right ) -  \left ({\bm {x}}^T, {\bm {y}}^T \right ) S \left ( \begin{array}{c} {\bm {x}} \\ {\bm {y}} \end{array} \right )  \right \}  \label{eq3.3}
\end{eqnarray}
for every ${\bm {\alpha}} ={\bm {x}} + i {\bm {y}},$ ${\bm {x}}, {\bm {y}} \in \mathbb{R}^n$ where ${\bm {\ell}},$ ${\bm {m}}$ are vectors in $\mathbb{R}^n$ and $S$ is a real positive definite matrix of order $2n$ satisfying the matrix inequality $2S - iJ \ge 0,$ with $J$ as in \eqref{eq1.2}.
\end{theorem}

\begin{proof}
We refer to the proof of Theorem 3.1 in \cite{krp}.\qed
\end{proof}

We remark that ${\bm {\ell}},$  ${\bm {m}}$ and $S$ in \eqref{eq3.3} are defined by the equations
\begin{eqnarray*}
{\bm {\ell}}^{T} {\bm {x}} - {\bm {m}}^T {\bm {y}} &=& \Tr \,\,\rho \sum_{j=1}^n (x_j p_j - y_j q_j) \\
({\bm {x}}^T, {\bm {y}}^T) S \left (\begin{array}{c} {\bm {x}} \\  {\bm {y}}  \end{array} \right ) &=& \Tr \,\,\rho \,\,X^2 - (\Tr \,\, \rho X)^2, X = \sum_{j=1}^n (x_j p_j - y_j q_j). 
\end{eqnarray*}
It is clear that $\ell_j$ is the expectation value of $p_j,$ $m_j$ is the expectation value of $q_j$ and $S$ is the covariance matrix of $p_1, p_2, \ldots, p_n;$ $-q_1, -q_2, \ldots, -q_n$ in the state $\rho$ defined by \eqref{eq3.3}. By a slight abuse of language we call $S$ the covariance matrix of the Gaussian state $\rho.$ All such Gaussian covariance matrices constitute the convex set $K_n$ defined already in \eqref{eq1.3}. We shall now investigate some properties of this convex set.

\begin{proposition}[Williamson's normal form \cite{dms}]\label{prop3.2}
Let $A$ be any real strictly positive definite matrix of order $2n.$ Then there exists a unique diagonal matrix $D$ of order $n$ with diagonal entries $d_1 \ge d_2 \ge \cdots \ge d_n > 0$ and a symplectic matrix $M$ in $Sp (2n, \mathbb{R})$ such that
\end{proposition}

\begin{equation}
A = M^T \,\,\left [\begin{array}{cc} D & 0 \\ 0 & D \end{array} \right ] \,M.  \label{eq3.4}
\end{equation}

\begin{proof}
Define
$$B = A^{1/2} \,\, J\,A^{1/2}$$
where $J$ is given by \eqref{eq1.2}. Then $B$ is a real skew symmetric matrix of full rank. Hence its eigenvalues, inclusive of multiplicity, can be arranged as $\pm id_1, \pm id_2, \ldots, \pm id_n$ where $d_1 \ge d_2 \ge \cdots \ge d_n > 0.$ Define $D = \diag (d_1, d_2, \ldots, d_n),$ i.e., the diagonal matrix with $d_i$ as the $ii$-th entry for $1 \leq i \leq n.$ Then there exists a real orthogonal matrix $\Gamma$ of order $2n$ such that
$$ \Gamma^T \,B\,\Gamma = \left [\begin{array}{cc} 0 & -D \\ D & 0 \end{array} \right ].$$
Define
$$L = A^{1/2} \,\Gamma \left [\begin{array}{cc} D^{-1/2}  & 0 \\ 0 & D^{-1/2} \end{array} \right ].$$

Then $L^T J L = J$ and
$$L A L^T =  \left [\begin{array}{cc} D & 0 \\ 0 & D \end{array} \right ].$$
Putting $M = (L^{-1})^T$ we obtain \eqref{eq3.4}.

To prove the uniqueness of $D,$ suppose $D^{\prime} = \diag (d_1^{\prime}, d_2^{\prime}, \ldots, d_n^{\prime})$ with $d_1^{\prime} \ge d_2^{\prime} \ge \cdots \ge d_n^{\prime} > 0$ and $M^{\prime}$ is another symplectic matrix of order $2n$ such that
$$A = M^T  \left [\begin{array}{cc} D & 0 \\ 0 & D \end{array} \right ] \,M = M^{\prime^{T}} \left [\begin{array}{cc} D^{\prime} & 0 \\ 0 & D^{\prime} \end{array} \right ] M^{\prime}.$$
Putting $N = MM^{^{\prime^{-1}}}$ we get a symplectic $N$ such that
$$N^T  \left [\begin{array}{cc} D & 0 \\ 0 & D \end{array} \right ]N = \left [\begin{array}{cc} D^{\prime} & 0 \\ 0 & D^{\prime} \end{array} \right ].$$ 
Substituting $N^T = JN^{-1} J^{-1}$ we get
$$N^{-1} \left [\begin{array}{cc} 0 & D \\ -D & 0 \end{array} \right ] N = \left [\begin{array}{cc} 0 & D^{\prime} \\ -D^{\prime} & 0 \end{array} \right ].$$
Identifying the eigenvalues on both sides we get $D = D^{\prime}$ \qed
\end{proof}

\begin{theorem}\label{thm3.3}
A real positive definite matrix $S$ is in $K_n$ if and only if there exists a diagonal matrix $D = \diag (d_1, d_2, \ldots, d_n)$ with $d_1 \ge d_2 \ge \cdots \ge d_n \ge \frac{1}{2}$ and a symplectic matrix $M \in Sp (2n, \mathbb{R})$ such that

\begin{equation}
S = M^T   \left [\begin{array}{cc} D & 0 \\ 0 & D \end{array} \right ] M.\label{eq3.5}
\end{equation}
In particular,

\begin{equation}
\det \,S = \prod_{i}^n d_j^2 \ge 4^{-n}.  \label{eq3.6}
\end{equation}
\end{theorem}

\begin{proof}
Let $S$ be a real strictly positive definite matrix in $K_n.$ From \eqref{eq1.3} we have $S \ge \frac{i}{2}J$ and therefore, for any $L \in Sp (2n, \mathbb{R}),$ 
\begin{equation}
L^T \,S\,L \ge \frac{i}{2}J.  \label{eq3.7}
\end{equation}
Using Proposition \ref{prop3.2} choose $L$ so that
$$L^T \,S\,L = \left [\begin{array}{cc} D & 0 \\ 0 & D \end{array} \right ]$$
where $D = \diag (d_1, d_2, \ldots, d_n),$ $d_1 \ge d_2 \ge \cdots \ge d_n > 0.$ Now \eqref{eq3.7} implies
$$\left [\begin{array}{cc} D & \frac{i}{2}I \\ -\frac{i}{2}I  & D \end{array} \right ] \ge 0. $$
The minor of second order in the left hand side arising from the $jj,$ $j\,n+j,$ $n+jj,$ $n+j \,\,n+j$ entries is $d_j^2 - \frac{1}{4} \ge 0.$ Choosing $L = M^{-1}$ we obtain \eqref{eq3.5} and \eqref{eq3.6}. Now we drop the assumption of strict positive definiteness on $S.$ From the definition of $K_n$ in \eqref{eq1.3} it follows that for any $S \in K_n$ one has $S + \varepsilon I \in K_n$ for every $\varepsilon > 0.$ Since $S+\varepsilon I$ is strictly positive definite $\det \,S+\varepsilon I \ge 4^{-n}$ $\forall$ $\varepsilon > 0.$ Letting $\varepsilon \rightarrow 0$ we see that \eqref{eq3.6} holds and $S$ is strictly positive definite.

To prove the converse, consider an arbitrary diagonal matrix $D = \diag (d_1, d_2, \ldots, d_n)$ with $d_1 \ge d_2 \ge \cdots \ge d_n \ge \frac{1}{2}.$ Clearly
$$2  \left [\begin{array}{cc} D & 0 \\ 0 & D \end{array} \right ] - i \left [\begin{array}{cc} 0 & -I \\ I & 0 \end{array} \right ] \ge 0,$$
and hence for any $M \in Sp (2n, \mathbb{R})$
$$2 M^T \left [\begin{array}{cc} D & 0 \\ 0 & D \end{array} \right ] M - i  \left [\begin{array}{cc} 0 & -I \\ I & 0 \end{array} \right ] \ge 0.$$
In other words,
$$M^T \left [\begin{array}{cc} D & 0 \\ 0 & D \end{array} \right ] M \in K_n \quad \,\,M \in Sp (2n, \mathbb{R}). $$
Finally, the uniqueness of the parameters $d_1 \ge d_2 \ge \cdots \ge d_n \ge \frac{1}{2}$ in the theorem is a consequence of Proposition \ref{prop3.2}. \qed
\end{proof}

We now prove an elementary lemma on diagonal matrices before the statement of our next result on the convex set $K_n.$

\begin{lemma}\label{lm3.4}
Let $D \ge I$ be a positive diagonal matrix of order $n.$ Then there exist positive diagonal matrices $D_1, D_2$ such that
$$D = \frac{1}{2} (D_1 + D_2) = \frac{1}{2} (D_1^{-1} + D_2^{-1}). $$
\end{lemma}

\begin{proof}
We write $D_2 = D_1 X$ and solve for $D_1$ and $X$ so that
$$2 D = D_1 (I+X) = D_1^{-1} (I + X^{-1}), $$
$D_1$ and $X$ being diagonal. Eliminating $D_1$ we get the equation
$$(I+X) (I+X^{-1}) = 4 D^2$$
which reduces to the quadratic equation
$$X^2 + (2 - 4 D^2) X + I = 0.$$
Solving for $X$ we do get a positive diagonal matrix solution
$$X  = I + 2 (D^2-1) + 2D (D^2-I)^{1/2}.$$
Writing 
$$D_1 = 2D (I+X)^{-1}, \quad D_2 = D_1 X$$
we get $D_1, D_2$ satisfying the required property.\qed
\end{proof}

\begin{theorem}\label{thm3.5}
A real positive definite matrix $S$ of order $2n$ belongs to $K_n$ if and only if there exist symplectic matrices $L,M$ such that
$$S = \frac{1}{4} (L^T L + M^T M).$$
Furthermore, $S$ is an extreme point of $K_n$ if and only if $S = \frac{1}{2} L^T L$ for some symplectic matrix $L.$
\end{theorem}

\begin{proof}
Let $S \in K_n.$ By Theorem \ref{thm3.3} we express $S$ as

\begin{equation}
S = N^T  \left [\begin{array}{cc} D & 0 \\ 0 & D \end{array} \right ] N \label{eq3.8}
\end{equation}
where $N$ is symplectic and $D = \diag (d_1, d_2, \ldots, d_n),$ $d_1 \ge d_2 \ge \cdots \ge d_n \ge \frac{1}{2}.$ Thus $2D \ge I$ and by Lemma \ref{lm3.4} there exist diagonal matrices $D_1 > 0,$ $D_2 > 0$ such that
$$2 D = \frac{1}{2} (D_1 + D_2) = \frac{1}{2} (D_1^{-1} + D_2^{-1}). $$
We rewrite \eqref{eq3.8} as
$$S = \frac{1}{4} N^T \left ( \left [\begin{array}{cc} D_1 & 0 \\ 0 & D_1^{-1} \end{array} \right ] + \left [\begin{array}{cc} D_2 & 0 \\ 0 & D_2^{-1} \end{array} \right ]\right )N.$$
Putting
$$ L = \left [\begin{array}{cc} D_1^{1/2} & 0 \\ 0 & D_1^{-1/2} \end{array} \right ] N, \quad M = \left [\begin{array}{cc} D_2^{1/2} & 0 \\ 0 & D_2^{-1/2} \end{array} \right ]$$ 
we have
$$S = \frac{1}{4} (L^T L + M^T M).$$
Since $\left [\begin{array}{cc} D_i^{1/2} & 0 \\ 0 & D_i^{-1/2} \end{array} \right ],$ $i = 1,2$ are symplectic it follows that $L$ and $M$ are symplectic. This proves the only if part of the first half of the theorem.

Since
$$\left [\begin{array}{cc} I & 0 \\ 0 & I \end{array} \right ] - i  \left [\begin{array}{cc} 0 & -I \\ I & 0 \end{array} \right ] \ge 0$$
multiplication by $L^T$ on the left and $L$ on the right shows that $L^TL - iJ \ge 0$ for any symplectic $L.$ Hence $\frac{1}{2} L^TL \in K_n$ $\forall$ $L \in Sp (2n, \mathbb{R}).$ Since $K_n$ is convex, $\frac{1}{4} (L^TL + M^T M) \in K_n,$ completing the proof of the first part.

The first part also shows that for an element $S$ of $K_n$ to be extremal it is necessary that $S = \frac{1}{2} L^T L$ for some symplectic $L.$ To prove sufficiency, suppose there exist $L \in Sp (2n, \mathbb{R}),$ $S_1, S_2 \in K_n$ such that 
$$\frac{1}{2} L^T L = \frac{1}{2} (S_1 + S_2).$$
By the first part of the theorem there exist $L_j \in Sp (2n, \mathbb{R})$ such that
\begin{equation}
L^T L = \frac{1}{4} \sum_{j=1}^4 L_j^T L_j  \label{eq3.9}
\end{equation}
where $S_1 = \frac{1}{4} (L_1^T L_1 + L_2^T L_2),$ $S_2 = \frac{1}{4} (L_3^T L_3 + L_4^T L_4).$ Left multiplication by $(L^T)^{-1}$ and right multiplication by $L^{-1}$ on both sides of \eqref{eq3.9} yields
\begin{equation}
I = \frac{1}{4} \sum_{j=1}^4 M_j  \label{eq3.10}
\end{equation}
where
$$M_j = (L^T)^{-1} L_j^T L_j L^{-1}. $$
Each $M_j$ is symplectic and positive definite. Multiplying by $J$ on both sides of \eqref{eq3.10} we get
\begin{eqnarray*}
J &=& \frac{1}{4} \sum_{j=1}^4 M_j J \\
&=& \frac{1}{4} \sum_{j=1}^4 M_j J M_j M_j^{-1} \\
&=& \frac{1}{4} J \sum_{j=1}^4 M_j^{-1}.
\end{eqnarray*}
Thus 
$$I = \frac{1}{4} \sum\limits_{j=1}^4 M_j = \frac{1}{4} \sum\limits_{j=1}^4 M_j^{-1} = \frac{1}{4} \sum\limits_{j=1}^4 \frac{1}{2} (M_j + M_j^{-1}),$$
which implies
$$\sum_{j=1}^4 \left (M_j^{1/2} - M_j^{-1/2} \right )^2 = 0, $$
or
$$M_j = I \quad \forall \quad 1 \leq j \leq 4$$
Thus
$$L_j^T L_j = L^T L \quad \forall \quad j$$
and $S_1 = S_2.$ This completes the proof of  sufficiency. \qed
\end{proof}

\begin{corollary}\label{cor3.6}
Let $S_1, S_2$ be extreme points of $K_n$ satisfying the inequality $S_1 \ge S_2.$ Then $S_1 = S_2.$
\end{corollary}

\begin{proof}
By Theorem \ref{thm3.5} there exist $L_i \in Sp (2n, \mathbb{R})$ such that $S_i = \frac{1}{2} L_i^T L_i,$ $i = 1,2.$ Note that $M = L_2 L_1^{-1}$ is symplectic and the fact that $S_1 \ge S_2$ can be expressed as $M^TM \leq I.$ Thus the eigenvalues of $M^TM$ lie in the interval $(0,1]$ but their product is equal to $(\det M)^2 = 1.$ This is possible only if all the eigenvalues are unity, i.e., $M^TM = I.$ This at once implies $L_1^T L_1 = L_2^T L_2.$ \qed
\end{proof}

Using the Williamson's normal form of the covariance matrix and the transformation properties of Gaussian states in Section \ref{sec:3} of \cite{krp} we shall now derive a formula for the density operator of a general Gaussian state. As in \cite{krp} denote by $\rho_g ({\bm {\ell}}, {\bm {m}}, S)$ the Gaussian state in $L^2 (\mathbb{R}^n)$ with the quantum Fourier transform
$$\rho_g ({\bm {\ell}}, {\bm {m}}, S)^{\wedge} ({\bm {z}}) = \exp - i\sqrt{2}({\bm {\ell}}^T {\bm {x}} - {\bm {m}}^T {\bm {y}}) - ({\bm {x}}^T {\bm {y}}^T) S \left (\begin{array}{c} {\bm {x}} \\ {\bm {y}} \end{array} \right ), {\bm {z}} = {\bm {x}} + i {\bm {y}}  $$
where ${\bm {\ell}},$ ${\bm {m}} \in \mathbb{R}^n$ and $S$ has the Williamson's normal form
$$S = M^T \left [\begin{array}{cc} D & 0 \\ 0 & D \end{array} \right ] M$$
with $M \in Sp(2n, \mathbb{R}),$ $D = \diag (d_1, d_2, \ldots, d_n),$ $d_1 \ge d_2 \ge \cdots \ge d_n \ge \frac{1}{2}.$ From Corollary 3.3 of \cite{krp} we have
$$W \left (\frac{{\bm {m}} + i {\bm {\ell}}}{\sqrt{2}} \right )^{\dagger} \rho_g ({\bm {\ell}}, {\bm {m}}, S)   W \left (\frac{{\bm {m}} + i {\bm {\ell}}}{\sqrt{2}} \right ) = \rho_g ({\bm {0}}, {\bm {0}}, S)$$
and Corollary 3.5 of \cite{krp} implies
$$\rho_g ({\bm {0}},{\bm {0}}, S ) = \Gamma (M)^{-1} \rho_g \left ({\bm {0}}, {\bm {0}}, \left [\begin{array}{cc} D & 0 \\ 0 & D \end{array} \right ] \right ) \Gamma (M). $$
Since $\left [\begin{array}{cc} D & 0 \\ 0 & D \end{array} \right ]$ is a diagonal covariance matrix
$$\rho_g \left ({\bm {0}}, {\bm {0}}, \left [\begin{array}{cc} D & 0 \\ 0 & D \end{array} \right ] \right ) = \bigotimes_{j=1}^n \rho_g (0.0, d_j I_2) $$
where the $j$-th component in the right hand side is the Gaussian state in $L^2 (\mathbb{R})$ with means $0$ and covariance matrix $d_j I_2, I_2$ denoting the identity matrix of order $2.$ If $d_j = \frac{1}{2}$ we have
$$\rho_g (0,0, \frac{1}{2} I_2) = | e (0)\rangle \langle e(0)| \,\,\,{\rm in}\,\, L^2(\mathbb{R}).$$

If $d_j > 1/2,$ writing $d_j = \frac{1}{2} \coth \frac{1}{2} s_j,$ one has
\begin{eqnarray*}
\rho_g (0,0, d_j I_2) &=& (1 - e^{-s_{j}}) e^{-s_{j} a^{\dagger} a} \\
&=& 2 \sinh \frac{1}{2} s_j \,\,e^{-\frac{1}{2} s_{j} (p^{2} + q^{2})} \,\, {\rm in} \,\, L^2 (\mathbb{R})
\end{eqnarray*}
with $a, a^{\dagger}, p, q$ denoting the operator of annihilation, creation, momentum and position respectively in $L^2 (\mathbb{R}).$ We now identify $L^2 (\mathbb{R}^n)$ and $L^2 (\mathbb{R})^{\otimes^{n}}$ and combine the reductions done above to conclude the following:

\begin{theorem}\label{thm3.7}
Let $\rho_g ({\bm {\ell}}, {\bm {m}}, S)$ be the Gaussian state in $L^2 (\mathbb{R}^n)$ with mean momentum and position vectors ${\bm {\ell}}, {\bm {m}}$ respectively and covariance matrix $S$ with Williamson's normal form
$$S = M^T \left [\begin{array}{cc} D & 0 \\ 0 & D \end{array} \right ] M, \quad M \in Sp (2n, \mathbb{R}),$$
$D = \diag (d_1, d_2, \ldots, d_n),$ $d_1 \ge d_2 \ge \cdots \ge d_m > d_{m+1} = d_{m+2} = \cdots = d_n = \frac{1}{2},$ $d_j = \frac{1}{2} \coth \frac{1}{2} s_j,$ $1 \leq j \leq m,$ $s_j > 0.$ Then
\begin{eqnarray}
\rho_g ({\bm {\ell}, {\bm {m}}, S)}  &=& W(\frac{{\bm {m}} + i {\bm {\ell}}}{\sqrt{2}}) \Gamma(M)^{-1} \prod_{j=1}^m (1 - e^{-s_{j}}) \times \nonumber \\
&&  e^{- \sum_{j=1}^{m} s_{j} a_{j}^{\dagger} a_{j}  } \otimes (|e(0) \rangle \times \langle e(0) |)^{\otimes ^{n-m}} \Gamma(M) W(\frac{{\bm {m}} + i {\bm {\ell}}}{\sqrt{2}})^{-1} \label{eq3.11}
\end{eqnarray}
where $W(\cdot)$ denotes Weyl operator, $\Gamma(M)$ is the unitary operator implementing the Bogolioubov automorphism of CCR corresponding to the symplectic linear transformation $M$ and $| e (0) \rangle$ denotes the exponential vector corresponding to $0$ in any copy of $L^2 (\mathbb{R}).$
\end{theorem} 

\begin{proof}
Immediate from the discussion preceding the statement of the theorem.\qed
\end{proof}

\begin{corollary}\label{cor3.8}
The wave function of the most general pure Gaussian state in$L^2 (\mathbb{R}^n)$ is of the form
$$| \psi \rangle = W({\bm {\alpha}}) \Gamma(U) \,\,|e_{\lambda_{1}} \rangle |e_{\lambda_{2}} \rangle  \cdots  |e_{\lambda_{n}} \rangle  $$
where
$$e_{\lambda} (x) = (2 \pi)^{-1/4} \lambda^{-1/2} \exp -4^{-1} \lambda^{-2} x^2, \quad x \in \mathbb{R}, \lambda > 0, $$
${\bm {\alpha}} \in \mathbb{C}^n,$ $U$ is a unitary matrix of order $n,$ $W ({\bm {\alpha}})$ is the Weyl operator associated with ${\bm {\alpha}},$ $\Gamma (U)$ is the second quantization unitary operator associated with $U$ and $\lambda_j,$ $1 \leq j \leq n$ are positive scalars.
\end{corollary}

\begin{proof}
Since the number operator $a^{\dagger} a$ has spectrum $\{0,1,2, \ldots \}$ it follows from Theorem \ref{thm3.7} that $\rho_g ({\bm {\ell}}, {\bm {m}}, S)$ is pure if and only if $m=0$ in \eqref{eq3.11}. This implies that the corresponding wave function $| \psi \rangle$ can be expressed as
\begin{equation}
| \psi \rangle = W ({\bm {\alpha}}) \Gamma (M)^{-1} (| e (0) \rangle)^{\otimes^{n}}  \label{eq3.12}
\end{equation}
where $M \in Sp (2n, \mathbb{R})$ and ${\bm {\alpha}} = \frac{{\bm {m}} + i {\bm {\ell}}}{\sqrt{2}}.$ The covariance matrix of this pure Gaussian state is $\frac{1}{2} M^T M.$ The symplectic matrix $M$ has the decomposition \cite{dms}
$$M = V_1  \left [\begin{array}{cc} D & 0 \\ 0 & D^{-1} \end{array} \right ] V_2$$
where $V_1$ and $V_2$ are real orthogonal as well as symplectic and $D$ is a positive diagonal matrix of order $n.$ Thus
\begin{eqnarray*}
M^T M &=& V_2^T \left [\begin{array}{cc} D^2 & 0 \\ 0 & D^{-2} \end{array} \right ] V_2 \\
&=& N^T N
\end{eqnarray*}
where
$$N = \left [\begin{array}{cc} D & 0 \\ 0 & D^{-1} \end{array} \right ] V_2.$$
Since the covariance matrix of $| \psi \rangle$ in \eqref{eq3.12} can also be written as $\frac{1}{2}N^T N,$ modulo a scalar multiple of modulus unity $| \psi \rangle$ can also be expressed as
\begin{equation}
 | \psi \rangle = W ({\bm {\alpha}}) \Gamma (V_2)^{-1} \Gamma \left ( \left [\begin{array}{cc} D^{-1} & 0 \\ 0 & D \end{array} \right ] \right ) \,\,|e(0) \rangle^{\otimes^{n}}.\label{eq3.13}
\end{equation}
If $U$ is the complex unitary matrix of order $n$ satisfying
\begin{eqnarray*}
U ({\bm {x}} + i {\bm {y}}) &=& {\bm {x}}^{\prime} + i {\bm {y}}^{\prime}, \\
\left [\begin{array}{c} {\bm {x}}^{\prime} \\ {\bm {y}}^{\prime} \end{array} \right ] &=& V_2^T  \left [\begin{array}{c} {\bm {x}} \\ {\bm {y}} \end{array} \right ] \quad \forall \quad  {\bm {x}}, {\bm {y}} \in \mathbb{R}^n
\end{eqnarray*}
and $D^{-1} = \diag (\lambda_1, \lambda_2, \ldots, \lambda_n)$ we can express \eqref{eq3.13} as
\begin{eqnarray*}
 | \psi \rangle &=& W ({\bm {\alpha}}) \Gamma (U) \left \{ \bigotimes_{j=1}^{n} \quad \Gamma \left ( \left [ \begin{array}{cc} \lambda_j & 0 \\ 0 & \lambda_j^{-1} \end{array} \right ] \right ) \,\,| e (0) \rangle \right \} \\
&=& W ({\bm {\alpha}}) \Gamma (U) \,|e_{\lambda_{1}} \rangle |e_{\lambda_{2}} \rangle \cdots |e_{\lambda_{n}} \rangle
\end{eqnarray*}
where we have identified $L^2 (\mathbb{R}^n)$ with $L^2 (\mathbb{R})^{\otimes^{n}}.$

We conclude this section with a result on the purification of Gaussian states.\qed
\end{proof}

\begin{theorem}\label{thm3.9}
Let $\rho$ be a mixed Gaussian state in $L^2 (\mathbb{R}^n).$ Then there  exists a pure Gaussian state $| \psi \rangle$ in $L^2 (\mathbb{R}^n) \otimes L^2 (\mathbb{R}^n)$ such that 
$$\rho = \Tr_2 \,\, U \,\,| \psi  \rangle \langle \psi | \,U^{\dagger} $$
for some unitary operator $U$ in  $L^2 (\mathbb{R}^n) \otimes L^2 (\mathbb{R}^n)$ with $\Tr_2$ denoting the relative trace over the second copy of $L^2 (\mathbb{R}^n).$
\end{theorem}

\begin{proof}
First we remark that by a Gaussian state in  $L^2 (\mathbb{R}^n) \otimes L^2 (\mathbb{R}^n)$ we mean it by the canonical identification of this product Hilbert space with $L^2 (\mathbb{R}^{2n}).$ Let $\rho = \rho_g ({\bm {\ell}}, {\bm {m}}, S) $ where by Theorem \ref{thm3.5} we can express
$$ S = \frac{1}{4}  \left (L_1^T L_1 + L_2^T L_2 \right ), \quad L_1, L_2 \in Sp(2n, \mathbb{R}). $$ 
Now consider the pure Gaussian states,
$$| \psi_{L_{i}} \rangle = \Gamma (L_i)^{-1} \,\,| e ({\bm {0}}) \rangle, \quad i = 1,2$$
in $L^2 (\mathbb{R}^n)$ and the second quantization unitary operator $\Gamma_0$ satisfying
$$ \Gamma_0 \,\,e ({\bm {u}} \oplus {\bm {v}}) = e \left (  \frac{{\bm {u}} + {\bm {v}} }{\sqrt{2}} \oplus \frac{{\bm {u}} -{\bm {v}}}{\sqrt{2}} \right ) \quad \forall \,\,{\bm {u}}, {\bm {v}} \in \mathbb{C}^n  $$
in $L^2 (\mathbb{R}^{2n})$ identified with $L^2 (\mathbb{R}^n) \otimes L^2 (\mathbb{R}^n),$ so that
$$e ({\bm {u}} \oplus {\bm {v}}  ) = e ({\bm {u}}) \otimes e ({\bm {v}} ).$$
Then by Proposition 3.11 of \cite{krp} we have
$$ \Tr_2 \,\,\Gamma_0 \left (    |\psi_{L_{1}} \rangle \langle \psi_{L_{1}}| \otimes  |\psi_{L_{2}} \rangle \langle \psi_{L_{2}}| \right ) \Gamma_0^{\dagger} = \rho_g ({\bm {0}}, {\bm {0}}, S).  $$
If ${\bm {\alpha}} = \frac{{\bm {m}} + i {\bm {\ell}}  }{\sqrt{2}}$ we have
$$W ({\bm {\alpha}}) \rho_g ({\bm {0}}, {\bm {0}}, S) W ({\bm {\alpha}})^{\dagger} =  \rho_g ({\bm {\ell}}, {\bm {m}}, S).$$
Putting 
$$U = \left ( W ({\bm {\alpha}}) \otimes I \right ) \Gamma_0 \left (\Gamma (L_1)^{-1} \otimes \Gamma (L_2)^{-1}  \right )$$
we get 
$$\rho_g \left ( {\bm {\ell}}, {\bm {m}},  S \right ) = \Tr_2 \,\, U \,\, | e ({\bm {0}}) \otimes e ({\bm {0}}) \rangle \langle   e ({\bm {0}}) \otimes e ({\bm {0}})| \, \, U^{\dagger}  $$
where $| e ({\bm {0}}) \rangle$ is the exponential vector in $L^2 (\mathbb{R}^n).$\qed
\end{proof}

\section{The symmetry group of the set of Gaussian states}
\label{sec:4}
\setcounter{equation}{0}

Let $S_n$ denote  the set of all Gaussian states in $L^2 (\mathbb{R}).$ We say that a unitary operator $U$ in $L^2 (\mathbb{R}^n)$ is a {\it Gaussian symmetry} if, for any $\rho \in S_n,$ the state $U \rho U^{\dagger}$ is also in $S_n.$ All such Gaussian symmetries constitute a group $\mathcal{G}_n.$ If ${\bm {\alpha}} \in \mathbb{C}^n$ and $L \in Sp (2n, \mathbb{R})$ then the associated Weyl operator $W ({\bm {\alpha}})$ and the unitary operator $\Gamma(L)$ implementing the Bogolioubov automorphism of CCR corresponding to $L$ are in $\mathcal{G}_n$ (See Corollary 3.5 in \cite{krp}.) The aim of this section is to show that any element $U$ in $\mathcal{G}_n$ is of the form $\lambda W({\bm {\alpha}}) \Gamma(L)$ where $\lambda$ is a complex scalar of modulus unity, ${\bm {\alpha}} \in \mathbb{C}^n$ and $L \in Sp(2n, \mathbb{R}).$ This settles a question raised in \cite{krp}.

We begin with a result on a special Gaussian state. 

\begin{theorem}\label{thm4.1}
Let $s_1 > s_2 > \cdots > s_n > 0$ be irrational numbers which are linearly independent over the field  $Q$ of rationals and let
$$\rho_{{\bm {s}}} = \rho_g ({\bm {0}}, {\bm {0}}, S) = \prod_{j=1}^n (1-e^{-s_{j}}) e^{- \sum\limits_{j=1}^{n} s_{j} a_{j}^{\dagger} a_{j}} $$
be the Gaussian state in $L^2(\mathbb{R}^n)$ with zero position and momentum mean vectors and covariance matrix
$$S =  \left [\begin{array}{cc} D & 0 \\ 0 & D \end{array} \right ], \quad D = \diag (d_1, d_2, \ldots, d_n)$$
with $d_j = \frac{1}{2} \coth \frac{1}{2} s_j.$ Then a unitary operator $U$ in $L^2 (\mathbb{R}^n)$ has the property that  $U \rho_{{\bm {s}}} U^{\dagger}$ is a Gaussian state if and only if, for some ${\bm { \alpha}} \in \mathbb{C}^n,$ $L \in Sp (2n, \mathbb{R})$ and a complex-valued function $\beta$ of modulus unity on $\mathbb{Z}_{+}^n$ 

\begin{equation}
U = W ({\bm {\alpha}}) \Gamma(L) \beta (a_1^{\dagger} a_1, a_2^{\dagger} a_2, \ldots, a_n^{\dagger} a_n)  \label{eq4.1}
\end{equation}
where $\mathbb{Z}_{+} = \{ 0,1,2, \ldots \}.$
\end{theorem}

\begin{proof}
Sufficiency is immediate from Corollary 3.3 and Corollary 3.5 of \cite{krp}. To prove necessity assume that
\begin{equation}
U \rho_{{\bm {s}}} U^{\dagger} = \rho_g ({\bm {\ell}}, {\bm {m}}, S^{\prime})  \label{eq4.2}
\end{equation}
Since $a^{\dagger} a$ in $L^2 (\mathbb{R})$ has spectrum $\mathbb{Z}_{+}$ and each eigenvalue $k$ has multiplicity one \cite{gcw} it follows that the selfadjoint positive operator $\sum\limits_{j=1}^n s_j a_j^{\dagger} a_j,$ being a sum of commuting self adjoint operators $s_j a_j^{\dagger} a_j,$ $1 \leq j \leq n$ has spectrum $\left \{\left .\sum\limits_{j=1}^n s_j k_j \right | k_j \in \mathbb{Z}_{+} \,\,\forall \,\, j \right \}$ with each eigenvalue of multiplicity one thanks to the assumption on $\{s_j, 1 \leq j \leq n \}.$ Since $\rho_{{\bm {s}}}$ and $U \rho_{{\bm {s}}} U^{-1}$ have the same set of eigenvalues and same multiplicities it follows from Theorem \ref{thm3.7} that 
\begin{equation}
U  \rho_{{\bm {s}}} U^{-1} = W ({\bm {z}}) \Gamma(M)^{-1} \rho_{{\bm {t}}} \Gamma(M) W ({\bm {z}})^{-1}  \label{eq4.3}
\end{equation}
where ${\bm {z}} \in \mathbb{C}^n,$ $M \in Sp (2n, \mathbb{R}),$ ${\bm {t}} = (t_1, t_2, \ldots, t_n)^T$ and 
$$\rho_{{\bm {t}}} = \prod_{j=1}^n (1-e^{-t_{j}}) e^{- \sum\limits_{j=1}^{n} t_{j} a_{j}^{\dagger} a_{j}}. $$
Since the maximum eigenvalues of $\rho_{{\bm {s}}}$ and $\rho_{{\bm {t}}}$ are same it follows that
$$\prod (1 - e^{-s_{j}}) = \prod (1-e^{-t_{j}}).$$
Since the spectra of $\rho_{{\bm {s}}}$ and $\rho_{{\bm {t}}}$  are same it follows that
$$\left \{\left .\sum_{j=1}^n s_j k_j \right | k_j \in \mathbb{Z}_{+} \quad \forall \,\, j \right \} = \left \{\left .\sum_{j=1}^n t_j k_j \right | k_j \in \mathbb{Z}_{+} \quad \forall \,\, j \right \}. $$
Choosing ${\bm {k}} = (0,0, \ldots, 0, 1, 0, \ldots, 0)^T$ with $1$ in the $k$-th position we conclude the existence of matrices $A,B$ of order $n \times n$ and entries in $\mathbb{Z}_{+}$ such that
$$ {\bm {t}} = A {\bm {s}}, \quad {\bm {s}} = B {\bm {t}}$$
so that  $BA {\bm {s}} = {\bm {s}}.$ The rationally linear independence of the $s_j$'s implies $BA=I.$ This is possible only if $A$ and $B = A^{-1}$ are both permutation matrices.

Putting $V = \Gamma(M) W ({\bm {z}})^{\dagger} U$ we have from \eqref{eq4.3}
$$V \rho_{{\bm {s}}} = \rho_{{\bm {t}}} V.$$
Denote by $| {\bm {k}} \rangle$ the vector satisfying
$$a_j^{\dagger} a_j \,\, | {\bm {k}} \rangle = k_j \,\,  | {\bm {k}} \rangle$$
where $| {\bm {k}} \rangle = | { {k}}_1 \rangle | { {k}}_2 \rangle \cdots | { {k}}_n \rangle.$ Then
\begin{eqnarray*}
V \rho_{{\bm {s}}} \,\, | {\bm {k}} \rangle &=& \prod_{j=1}^n (1 - e^{-s_{j}}) e^{- \sum s_{j} k_{j}} \,\,V \,\,| {\bm {k}} \rangle \\
&=& \rho_{{\bm {t}}} \, V | {\bm {k}} \rangle, \quad {\bm {k}} \in \mathbb{Z}_{+}^n.
\end{eqnarray*}
Thus $V \,| {\bm {k}} \rangle$ is an eigenvector for $\rho_{{\bm {t}}}$ corresponding to the eigenvalue
\begin{eqnarray*}
\prod (1 -e^{-s_{j}}) e^{- {\bm {s}}^{{\bm {t}}} {\bm {k}}} &=& \prod_{j=1}^n (1-e^{-t_{j}}) e^{-{\bm {t}}^{T} B^{T} {\bm {k}}} \\
&=& \prod_{j=1}^n (1-e^{-t_{j}}) e^{-{\bm {t}}^{T} A {\bm {k}}}.
\end{eqnarray*}
Hence there exists a scalar $\beta ({\bm {k}})$ of modulus unity such that
\begin{eqnarray*}
V \,\, | {\bm {k}} \rangle &=& \beta ({\bm {k}}) \,\, | A {\bm {k}} \rangle \\
&=& \Gamma(A) \beta (a_1^{\dagger} a_1, a_2^{\dagger} a_2, \ldots, a_n^{\dagger} a_n) \,\| {\bm {k}} \rangle \,\,\,\forall \,\,\, {\bm {k}} \in \mathbb{Z}_{+}^n.
\end{eqnarray*}
where $\Gamma(A)$ is the second quantization of the permutation unitary matrix $A$ acting in $\mathbb{C}^n.$ Thus
$$U = W ({\bm {z}}) \Gamma(M)^{\dagger} \Gamma(A) \beta (a_1^{\dagger} a_1, a_2^{\dagger} a_2, \ldots, a_n^{\dagger} a_n).$$
which completes the proof.\qed
\end{proof}

\begin{theorem}\label{thm4.2}
  A unitary operator $U$ in $L^2 (\mathbb{R}^n)$ is a Gaussian symmetry if and only if there exist a scalar $\lambda$ of modulus unity, a vector ${\bm {\alpha}}$ in $\mathbb{C}^n$ and a symplectic matrix $L \in Sp (2n, \mathbb{R})$ such that
$$U = \lambda W ({\bm {\alpha}}) \Gamma (L)$$
where $W ({\bm {\alpha}})$ is the Weyl operator associated with ${\bm {\alpha}}$ and $\Gamma(L)$ is a unitary operator implementing the Bogolioubov automorphism of CCR corresponding to $L.$
\end{theorem}

\begin{proof}
The if part is already contained in Corollary 3.3 and Corollary 3.5 of \cite{krp}. In order to prove the only if part we may, in view of Theorem \ref{thm4.1}, assume that $U = \beta (a_1^{\dagger} a_1, a_2^{\dagger} a_2, \ldots, a_n^{\dagger} a_n)$ where $\beta$ is a function of modulus unity on $\mathbb{Z}_{+}^n.$ If such a $U$ is a  Gaussian symmetry then, for any pure Gaussian state $| \psi \rangle,$ $U | \psi \rangle$ is also a pure Gaussian state. We choose
$$|\psi \rangle = e^{-\frac{1}{2} \|{\bm {u}} \|^{2}} |e ({\bm {u}})\rangle = W ({\bm {u}}) | e({\bm {0}}) \rangle $$
where ${\bm {u}} = (u_1, u_2, \ldots, u_n)^T \in \mathbb{C}^n$ with $u_j \neq 0 \,\,\forall \,\,j.$ By our assumption
\begin{equation}
| \psi^{\prime} \rangle = e^{-\frac{1}{2} \|{\bm {u}} \|^{2}} \beta (a_1^{\dagger} a_1, a_2^{\dagger} a_2, \ldots, a_n^{\dagger} a_n) |e({\bm {u}}) \rangle \label{eq4.4}
\end{equation}
is also a pure Gaussian state. By Corollary \ref{cor3.8}, $\exists {\bm {\alpha}} \in \mathbb{C}^n,$ a unitary matrix $A$ of order $n$ and $\lambda_j >0,$ $1 \leq j \leq n$ such that 
 \begin{equation}
| \psi^{\prime} \rangle =  W ({\bm {\alpha}}) \Gamma(A) \,|e_{\lambda_{1}} \rangle  |e_{\lambda_{2}} \rangle \cdots  |e_{\lambda_{n}} \rangle. \label{eq4.5}
\end{equation}
Using \eqref{eq4.4}  and \eqref{eq4.5} we shall evaluate the function $f ({\bm {z}}) = \langle \psi^{\prime} | e({\bm {z}}) \rangle$  in two different ways. From \eqref{eq4.4} we have
\begin{eqnarray}
f ({\bm {z}}) &=& e^{-\frac{1}{2} \|{\bm {u}} \|^{2}} \langle e({\bm {u}}) \left | \bar{\beta} (a_1^{\dagger} a_1, a_2^{\dagger} a_2, \ldots, a_n^{\dagger} a_n)\right | e({\bm {z}}) \rangle \nonumber \\
&=&  e^{-\frac{1}{2} \|{\bm {u}} \|^{2}}  \sum_{{\bm{z}} \in \mathbb{Z}_{+}^{n}}  \frac{\bar{\beta} (k_1, k_2, \ldots, k_n)}{k_1! k_2! \ldots k_n!} (\bar{u}_1 z_1)^{k_{1}} \ldots (\bar{u}_n z_n)^{k_{n}} \,| k_1 k_2 \cdots k_n \rangle \label{eq4.6}
\end{eqnarray}
where $| k_1 k_2 \cdots k_n \rangle = | k_1 \rangle | k_2 \rangle \cdots| k_n \rangle$ and $| e(z) \rangle = \sum\limits_{k \in \mathbb{Z}_{+}} \frac{z^k}{\sqrt{k!}} |k \rangle $ for $z \in \mathbb{C}.$

Since $| \beta ({\bm {k}})| = 1,$ \eqref{eq4.6} implies
 \begin{equation}
|f ({\bm {z}})| \leq \exp \left \{-\frac{1}{2} \|{\bm {u}}\|^2 + \sum_{j=1}^n \,\, |u_j| \,\,|z_j| \right \}.  \label{eq4.7}
\end{equation}

From the definition of $|e_{\lambda}\rangle$ in Corollary \ref{cor3.8} and the exponential vector $|e(z) \rangle$ in $L^2 (\mathbb{R})$ one has
$$\langle e_{\lambda} | e(z) \rangle = \sqrt{\frac{2 \lambda}{1+\lambda^2}} \,\exp \,\frac{1}{2} \left (\frac{\lambda^2 -1}{\lambda^2+1} \right ) z^2, \quad \lambda > 0, \quad z \in \mathbb{C}. $$
This together with \eqref{eq4.5} implies
\begin{eqnarray*}
 f({\bm {z}}) &=& \langle \left . e_{\lambda_{1}} \otimes e_{\lambda_{2}} \otimes \cdots \otimes e_{\lambda_{n}} \,\, \right | \Gamma(A^{-1}) W (-{\bm {\alpha}}) e ({\bm {z}})  \rangle \\
&=& e^{\langle {\bm {\alpha}} | {\bm {z}}\rangle - \frac{1}{2} \| {\bm {\alpha}}\|^{2}} \langle \left . e_{\lambda_{1}} \otimes e_{\lambda_{2}} \otimes \cdots \otimes e_{\lambda_{n}} \,\, \right | e (A^{-1} ({\bm {z}} + {\bm {\alpha}}))  \rangle 
\end{eqnarray*}
which is a nonzero scalar multiple of the exponential of a polynomial of degree 2 in $z_1, z_2, \ldots, z_n$ except when all the $\lambda_j$'s are equal to unity. This would contradict the inequality \eqref{eq4.6} except when $\lambda_j = 1 \,\,\forall \,\, j.$ Thus $\lambda_j = 1 \,\,\forall \,\, j$ and \eqref{eq4.5} reduces to
\begin{eqnarray*}
| \psi^{\prime} \rangle &=& W ({\bm {\alpha}}) \Gamma(A) \,| e ({\bm {0}} \rangle \\
&=& e^{-\frac{1}{2} \|{\bm {\alpha}} \|^{2}} \,|e ({\bm {\alpha}})\rangle .
\end{eqnarray*}
Now \eqref{eq4.4} implies
\begin{eqnarray*}
\lefteqn{\beta (a_1^{\dagger} a_1, a_2^{\dagger} a_2, \ldots, a_n^{\dagger} a_n) \,|e ({\bm {u}}) \rangle} \\
&=& e^{\frac{1}{2} (\|{\bm {u}} \|^{2} - \|{\bm {\alpha}} \|^{2})} \,| e ({\bm {\alpha}}) \rangle,
\end{eqnarray*}
or
\begin{eqnarray*}
\lefteqn{\sum_{{\bm {k}} \in \mathbb{Z}_{+}^{n}} \frac{u_1^{k_{1}} u_2^{k_{2}} \ldots   u_n^{k_{n}} }{\sqrt{k_1!} \cdots \sqrt{k_n!}}   \beta (k_1, k_2, \ldots, k_n) \,|k_1 k_2 \ldots k_n  \rangle     } \\
&=& e^{\frac{1}{2} (\|{\bm {u}} \|^{2} - \|{\bm {\alpha}} \|^{2})}  \sum \frac{\alpha_1^{k_{1}} \alpha_2^{k_{2}} \ldots   \alpha_n^{k_{n}} }{\sqrt{k_1!} \cdots \sqrt{k_n!}}  \,| k_1 k_2 \ldots k_n \rangle.
\end{eqnarray*}
Thus
$$\beta (k_1, k_2, \ldots, k_n) = e^{\frac{1}{2} (\|{\bm {u}} \|^{2} - \|{\bm {\alpha}} \|^{2})} \left (\frac{\alpha_1}{u_1} \right )^{k_{1}} \cdots \left ( \frac{\alpha_n}{u_n} \right )^{k_{n}}. $$
Since $|\beta ({\bm {k}})| = 1$ and $u_j \neq 0 \,\forall \,j$ it follows that $|\frac{\alpha_j}{u_j}|=1$ and
$$ \beta ({\bm {k}}) = e^{i \sum\limits_{j=1}^{n} \theta_{j} k_{j}} \quad \forall \,\,{\bm {k}} \in \mathbb{Z}_{+}^{n} $$
where $\theta_j$'s are real. Thus $\beta (a_1^{\dagger} a_1, a_2^{\dagger} a_2, \ldots, a_n^{\dagger} a_n) = \Gamma (D),$ the second quantization of the diagonal unitary matrix $D= \diag (e^{i \theta_{1}},e^{i \theta_{2}}, \ldots,  e^{i \theta_{n}}).$ This completes the proof.\qed
\end{proof}

\end{document}